\documentclass[preprint,11pt,times]{elsarticle}

\usepackage{amsfonts,amsmath}
\usepackage{graphicx}
\usepackage{algorithm,algorithmic}

\newcommand{\ZZ}{\mathbb{Z}}
\newcommand{\RR}{\mathbb{R}}
\newcommand{\NN}{\mathbb{N}}
\newcommand{\CC}{\mathbb{C}}

\newcommand{\rank}{{\mathop{\rm rank\,}\nolimits}}
\newcommand{\diag}{{\mathop{\rm diag\,}\nolimits}}

\newtheorem{theorem}{Theorem}[section]
\newtheorem{lem}[theorem]{Lemma}

\newtheorem{rem}[theorem]{Remark}

\newenvironment{proof}
{\medskip \noindent{\bf Proof:} \hspace{0.03cm}}{{$\Box$}}

\journal{Linear Algebra and Its Applications}

\begin{document}

\begin{frontmatter}
  
\title{SVD update methods for large matrices and
  applications\tnoteref{ded,fund}}

\tnotetext[ded]{Dedicated to Mariano Gasca on the occasion of his 75th
birthday with friendship and gratitude for long years of support and fruitful
collaboration.}
\tnotetext[fund]{This work was partially supported by the 
Spanish Research Grant MTM2015-65433-P (MINECO/FEDER), by Gobierno de
Arag\'on and Fondo Social Europeo.}

\author[jmp]{Juan Manuel Pe\~na}
\ead{jmpena@unizar.es}
\author[ts]{Tomas Sauer}
\ead{Tomas.Sauer@uni-passau.de}
\address[jmp]{Depto. Matem\'{a}tica Aplicada,
    Fac. Ciencias, Universidad de Zaragoza, E--50009 Zaragoza,
    SPAIN.}
\address[ts]{Lehrstuhl Mathematik mit Schwerpunkt Digitale
    Bildverarbeitung \& FORWISS, Universit\"at Passau, Fraunhofer IIS
    Reserch Group for Knowledge Based Image Processing, Innstr. 43,
    94032 Passau, GERMANY.}
  

\begin{abstract}
  We consider the problem of updating the SVD when augmenting a ``tall
  thin'' matrix, i.e., a rectangular matrix $A \in \RR^{m \times n}$ with $m
  \gg n$. Supposing that an SVD of $A$ is already known, and given a
  matrix $B \in \RR^{m \times n'}$, we derive an efficient method to 
  compute and efficiently store the SVD of the augmented matrix $[ A B
  ] \in \RR^{m \times (n+n')}$. This is an important tool for two types
  of applications: in the context of principal component analysis, the
  dominant left singular vectors provided by this decomposition form
  an orthonormal basis for the best linear subspace of a given
  dimension, while from the right singular vectors one can extract an
  orthonormal basis of the kernel of the matrix.
  We also describe two concrete applications of these concepts which
  motivated the development of our method and to which it is very well
  adapted.
\end{abstract}


\end{frontmatter}
  
\section{Introduction}
The \emph{singular value decomposition} $A = U \Sigma V^T$ of a
matrix $A \in \RR^{m \times n}$ is a useful and important tool in many
applications and there exist algorithms and even toolboxes to perform
this task numerically. Indeed, the
matrices $U$, $V$, obtained in the decomposition give valuable
information about $A$ in a compressed way: the \emph{singular
  vectors}, the columns of the matrix $U$, give best low
dimensional approximations of the subspace generated by the columns of
$A$, while the matrix $V$ reliably detects the \emph{kernel} of the
matrix $A$ and therefore is the method of choice for numerical rank
detection, see \cite{eaton09:_gnu_octav,GolubvanLoan96}.

In many applications, specifically in video processing or in the
multivariate versions of Prony's method, the full matrix $A$ is not
known from the beginning,
but is built by successively adding columns or blocks of
columns; these added blocks can correspond to new measurements or can just
be determined by the algorithm itself. What these applications have
in common is that the columns of the matrices are large, i.e., $m
\gg n$. This calls for reliable and efficient \emph{update algorithms}
for a matrix whose SVD is already known. Such an algorithm, due to
Businger \cite{businger70:_updat}, is described in \cite{Bjoerck96},
however this algorithm, more precisely its transposed version to add
columns, then assumes that $m \le n$ which means that it adds
mostly redundant columns. In \cite{brand06:_fast}, Brand more recently
gave a fast algorithm to update few dominant singular values of an
augmented matrix which was used, for example, to perform background
elimination in videos
\cite{rodriguez16:_increm_princ_compon_pursuit_video_backg_model}. Brand's
method is based on an efficient way to perform rank $1$ modifications, see
\cite[Sections~3 and 4.1]{brand06:_fast}, hence it would proceed columnwise to
add a full block.
In this paper, we give an algorithm
that adds the columns simultaneously to a ``tall and thin'' matrix,
but always computes the \emph{full} SVD in an efficient way, and point
out the connections to PCA based video analysis and Prony's method. 

In the latter, 
the main numerical problem in the algorithms presented
in \cite{Sauer17_Prony,Sauer18:_prony} is to determine reliably the
nullspaces 
of sequences of matrices that are generated by successively adding
blocks of columns. Adding these columns corresponds to extending a
symmetric H--basis for an ideal, and treating them in a symmetric way
and not attaching them in some order is of great importance for the
numerical performance of such algorithms.
In other words, we have to determine the nullspaces of
sequences
$$
B_0, \quad [ B_0 \, B_1 ], \quad [ B_0 \, B_1 \, B_2 ], \dots,
$$
where, depending on the algorithm used, $B_j$ can be a single column
or a block of several columns. Clearly, the rank of these matrices is
increasing which is, however, not captured by a naive application of
\texttt{Matlab}'s \texttt{rank} command to the augmented
matrices. Since numerical rank computations are usually based on a
singular value decomposition (SVD) of the matrix, we aim for a method
which uses an already existing SVD of a given matrix $A_k$ to compute
the SVD of the column augmented $A_{k+1} = [ A_k \, B_k ]$ in an
efficient, numerically stable and reliable way.

Though the concrete method we develop and investigate here is new, the
problem itself has been considered before. Indeed,
Updating methods for rank revealing factorizations have been
considered by Stewart \cite{stewart93:_updat} in a very similar
context, namely for the \emph{MUSIC} algorithm
\cite{schmidt86:_multip} that solves Prony's problem in one variable
in the context of multisource radar signal processing. Other approaches for
updating SVD and QR algorithm can be found in \cite{bunch78:_updat}
and \cite{daniel76:_reort_gram_schmid_qr}, respectively. Incremental
methods for dominant singular subspaces were also considered in
\cite{baker12:_low_rank}, where only some dominant singular vectors
were computed. In contrast to that, our approach aims to always
compute the full \emph{thin} SVD of the matrix which is especially
needed for kernel computations.

The layout of the paper is as follows. In Section 2 we present our
method to update the SVD and analyze its computational cost. Section 3
presents a corresponding thresholding strategy, which is designed
ensure that increasing ranks are detected properly.
Section 4 presents two applications for which the augmented SVD
method is very well suited: to Principal Component Analysis of videos
in Subsection 4.3 and to the already mentioned solution of Prony's 
problem in several variables, see Subsection 4.4.

\section{Updating decompositions -- idea and details}
\label{sec:Upd}
We are considering processes that determine matrices $A_k \in \RR^{d
  \times n_k}$, $k=1,2,\dots$, by the iterative block column extension
\begin{equation}
  \label{eq:UpdateSetting}
  A_{k+1} = \left[ A_k \, B_k \right], \qquad B_k \in \RR^{d \times
    m_k}, \qquad m_k := n_{k+1} - n_k,  
\end{equation}
where $n_k \le d$, usually $n_k \ll d$,
and want to compute a rank revealing decomposition or an SVD for $A_{k+1}$
in an efficient and numerically stable way from that of
$A_k$.

We begin with the SVD and adapt an idea to our needs which is
referenced in \cite{Bjoerck96}, as Businger's method
\cite{businger70:_updat}.
The exposition in \cite{Bjoerck96}, however, extends a matrix with more rows
than columns by adding a further row and it is mentioned in passing
that adding of columns can be done by transposition. Then, however,
the matrix should have more columns than rows which is not the case in
our situation. Nevertheless, the basic idea can be adapted.

To that end, we assume that we already computed a decomposition
\begin{equation}
  \label{eq:AkSVD}
  A_k = U_k \Sigma_k V_k^T, \qquad \rank A_k =: r_k \le n_k \le d,
\end{equation}
with orthogonal matrices
\begin{equation}
  \label{eq:AkSVDUV}
  U_k \in \RR^{d \times d}, \qquad V_k \in \RR^{n_k \times n_k},
\end{equation}
and the diagonal matrix
\begin{equation}
  \label{eq:AkSVDSigma}
  \Sigma_k = \left[
    \begin{array}{cc}
      \Sigma_k' & 0_{r_k \times (n_k - r_k)} \\
      0_{(d-r_k) \times r_k} & 0_{(d-r_k )\times (n_k - r_k)}
    \end{array}
  \right] \in \RR^{d \times n_k}, \qquad \Sigma_k' \in \RR^{r_k \times r_k},
\end{equation}
where $\Sigma_k'$ has strictly positive diagonal values.

\begin{rem}
  Due to \eqref{eq:AkSVDUV}, the factorization is formally \emph{not}
  a ``slim'' or ``economic'' decomposition of $A_k$. Such a
  decomposition would be of the form
  $$
  A_k = U \Sigma V^T, \qquad U \in \RR^{d \times r_k}, \, \Sigma \in
  \RR^{r_k \times r_k}, \, V \in \RR^{r_k \times n_k},
  $$
  with all diagonal elements of $\sigma$ being positive.

  Note,
  however, that in \eqref{eq:AkSVD} the last $d-r_k$ columns of $U$
  and the last $n_k - r_k$ columns ov $V$ are irrelevant for the
  validity of the decomposition, hence it is not unique. We will later
  describe how to 
  represent \emph{one} such decomposition with a memory effort that only
  exceeds that of a thin representation by $r_k (r_k + n_k)$
  elements. This is negligible in the case when $d \gg n_k$ and has
  the advantage that we always compute a full orthonormal basis of the
  kernel of $A_k$, which was motivated by its importance for the Prony
  application.
\end{rem}

\noindent
In what follows, we first deduce the updating method in a general
fashion and give the numerically efficient implementations of the
crucial steps afterwards. To that end, we define in a
straightforward way, 
$$
X := U_k^T A_{k+1} \left[
  \begin{array}{cc}
    V_k & 0 \\ 0 & I
  \end{array}
\right] = U_k^T [ A_k \, B_k ] \left[
  \begin{array}{cc}
    V_k & 0 \\ 0 & I
  \end{array}
\right] = \left[ \Sigma_k \,|\, U_k^T B_k \right]
$$
and observe that
\begin{equation}
  \label{eq:XDecompY}
  X = \left[
    \begin{array}{ccc}
      \Sigma_k' & 0 & * \\
      0 & 0 & *
    \end{array}
  \right] = \left[
    \begin{array}{ccc}
      \Sigma_k' & * & 0 \\
      0 & * & 0
    \end{array}
  \right] \, P_k =:  \left[
    \begin{array}{ccc}
      \Sigma_k' & Y' & 0 \\
      0 & Y & 0
    \end{array}
  \right] \, P_k
\end{equation}
for the permutation matrix
$$
P_k = \left[
  \begin{array}{ccc}
    I_{r_k} & 0 & 0 \\
    0 & 0 & I_{n_k - r_k} \\
    0 & I_{m_k} & 0 
  \end{array}
\right] \in \RR^{n_{k+1} \times n_{k+1}},
$$
that permutes the last $n_{k+1} - r_k$ columns.
Next, we apply a QR method with column pivoting on the matrix $Y \in
\RR^{(d-r_k) \times m_k}$, finding a permutation $P \in \RR^{m_k \times
  m_k}$ and a decomposition
\begin{equation}
  \label{eq:YQRP}
  Y = \widetilde Q \, \left[
    \begin{array}{c}
    \widetilde R \\
    0 \end{array}
  \right] 
  \, P, \qquad \widetilde Q \in
  \RR^{(d-r_k) \times (d-r_k)}, \quad \widetilde R \in \RR^{m_k \times m_k},
\end{equation}
where, as usually in rank revealing factorizations, column pivoting
ensures that the entries of $\widetilde R= \left[ \tilde r_{ij} : i,j
  = 1,\dots,m_k \right]$ satisfy
\begin{equation}
  \label{eq:rDescent}
  |\tilde r_{11}| \ge \cdots \ge |\tilde r_{m_k,m_k}| \qquad
  \text{and} \qquad
  |\tilde r_{jj}| \ge 
  |\tilde r_{j\ell}|, \quad \ell \ge j, \, j=1,\dots,m_k.
\end{equation}
Defining the orthogonal matrix
$$
Q = \left[
  \begin{array}{cc}
    I_{r_k \times r_k} & 0 \\
    0 & \widetilde Q
  \end{array}
\right] \in \RR^{d \times d},
$$
the decomposition (\ref{eq:YQRP}) yields that
\begin{equation}
  \label{eq:YQR}
  Q^T \left[
    \begin{array}{c}
      Y' \\ Y
    \end{array}
  \right] = \left[
    \begin{array}{c}
      Y' \\ \widetilde Q^TY
    \end{array}
  \right] = 
  \left[
    \begin{array}{c}
      Y' P^T \\
      \widetilde R \\
        0 
      \end{array}
  \right] P.
\end{equation}
The computation of the matrix $\widetilde R$ from (\ref{eq:YQRP}) is also
the starting point for a thresholding algorithm to be described in
the next section. Substituting (\ref{eq:YQR}) into (\ref{eq:XDecompY})
we then also get
\begin{equation}
  \label{eq:XPermuts}
  Q^T X = \left[
    \begin{array}{ccc}
      \Sigma_k' & Y' P^T & 0 \\
      0 & \widetilde R & 0\\
      0&0&0 
    \end{array}
  \right] \left[
    \begin{array}{ccc}
      I_{r_k} \\
      & P \\
      & & I_{n_k - r_k}
    \end{array}
  \right] P_k = \left[
    \begin{array}{ccc}
      \Sigma_k' & Y' P^T & 0 \\
      0 & \widetilde R & 0 \\
      0&0&0  
    \end{array}
  \right] P_k' P_k,
\end{equation}
with the block diagonal permutation
$$
P_k' := 
\left[
  \begin{array}{ccc}
    I_{r_k} && \\
             & P \\
             & & I_{n_k - r_k}
  \end{array}
\right]
$$
that satisfies
$$
P_k' P_k = \left[
  \begin{array}{ccc}
    I && \\
             & P \\
             & & I
  \end{array}
\right] \left[
  \begin{array}{ccc}
    I & & \\
      & & I \\
      & I &  
  \end{array}
\right] = \left[
  \begin{array}{ccc}
    I & & \\
      & & P \\
      & I &  
  \end{array}
\right].
$$
Therefore, using the abbreviation $p_k := r_k + m_k$, $r_{k+1} \le p_k
\le n_{k+1}$, we
obtain another upper triangular matrix of relatively small size:
\begin{equation}
  \label{eq:Rk+1Def}
  Q^T X (P_k' P_k)^T = \left[
    \begin{array}{ccc}
      \Sigma_k' & Y' P^T & 0 \\
      0 & \widetilde R & 0 \\
      0 & 0 & 0
    \end{array}
  \right]
  =: \left[
    \begin{array}{cc}
      R_{k+1} & 0 \\
      0 & 0
    \end{array}
  \right], \qquad R_{k+1} \in \RR^{p_k \times p_k}.
\end{equation}
Next, we compute a singular value decomposition of $R_{k+1}$ as
\begin{equation}
  \label{eq:Rk+1SVD}
  R_{k+1} = \widetilde U \, \left[
  \begin{array}{cc}
    \Sigma_{k+1}' & 0 \\
    0 & 0
  \end{array}
\right]
\widetilde V^T, \qquad
\qquad \widetilde U, \widetilde V \in \RR^{p_k \times p_k}, \;
\Sigma_{k+1}' \in \RR^{r_{k+1} \times r_{k+1}},
\end{equation}
where $\Sigma_{k+1}'$ has \emph{strictly positive} singular values
that can be controlled by means of the thresholding strategies
in the next section. This also determines the rank $r_{k+1}$ of $A_{k+1}$.

These results can be recombined into an efficient SVD of the $d \times
n_{k+1}$ matrix
\begin{align}
  \nonumber
  Q^T & X (P_k' P_k)^T \\
  \label{eq:overallDecomp}
      &= \left[
  \begin{array}{cc}
    \widetilde U & 0 \\ 0 & I_{d-p_k} \\
  \end{array}
\right] \left[
  \begin{array}{cc}
    \Sigma_{k+1}' & 0 \\ 0 & 0_{(d-r_{k+1}) \times (n_{k+1} - r_{k+1})}
  \end{array}
\right] \left[
  \begin{array}{cc}
    \widetilde V^T & 0 \\ 0 & I_{n_{k+1} - p_k}
  \end{array}
\right].
\end{align}
Since $p_k = r_k + m_k \le n_{k+1}$ with
equality iff $r_k = n_k$, i.e., iff $A_k$ has full rank, we can always
assume that
$$
\widetilde U, \widetilde V \in \RR^{n_{k+1} \times n_{k+1}},
$$
so that the storage requirement for these matrices is at most $n_{k+1}^2$.
Note, however, that the matrices in (\ref{eq:overallDecomp}) are now
patterned in different ways and 
that only their overall dimensions coincide.
Combining all decompositions, finally gives
\begin{eqnarray*}
  \lefteqn{
    A_{k+1} = U_k X \left[
      \begin{array}{cc}
        V_k^T & 0 \\ 0 & I 
      \end{array}
    \right] } \\
  & = & U_k Q \left[
    \begin{array}{cc}
      \widetilde U & 0 \\
      0 & I
    \end{array}
  \right] \left[
    \begin{array}{cc}
      \Sigma_{k+1}' & 0 \\
      0 & 0
    \end{array}
  \right] \left[
    \begin{array}{cc}
      \widetilde V^T & 0 \\
      0 & I
    \end{array}
          \right] \, P_k' P_k \,
          \left[
          \begin{array}{cc}
    V_k^T & 0 \\ 0 & I \\
    \end{array}
  \right] \\
  & =: & U_{k+1} \Sigma_{k+1} V_{k+1}^T,
\end{eqnarray*}
with the \emph{update rules}
\begin{eqnarray}
  \label{eq:SVDUpdateU}
  U_{k+1} & = & U_k Q \left[
    \begin{array}{cc}
      \widetilde U & 0 \\ 0 & I
    \end{array}
  \right], \\
  \label{eq:SVDUpdateV}
  V_{k+1} & = & \left[
    \begin{array}{cc}
      V_k& 0 \\ 0 & I
    \end{array}
  \right] (P_k' P_k)^T
  \left[
    \begin{array}{cc}
      \widetilde V & 0 \\ 0 & I 
    \end{array}
  \right].
\end{eqnarray}
The matrices appearing in (\ref{eq:SVDUpdateU}) are all of dimension
$d \times d$, the ones used in (\ref{eq:SVDUpdateV}) of dimension
$n_{k+1} \times n_{k+1}$.

\begin{rem}
  The re-computation of all singular values in (\ref{eq:Rk+1SVD}) is
  unavoidable since, for example the interlacing property of singular
  values of an augmented matrix tells us that usually all the singular
  values will change.
\end{rem}


\begin{rem}
  Businger's method as described in \cite{Bjoerck96} and also the
  efficient methods in \cite{brand06:_fast} treat only the case of
  adding a single row to the matrix. In the algebraic applications,
  especially in the multivariate version of Prony's method, however,
  it is important to treat the addition of several columns at the same
  time and treat these columns as symmetric as possible.
\end{rem}

\noindent
Since in the applications we consider, the
dimension $d$ and
therefore the size of $U_k \in \RR^{d \times d}$ can be rather large, it is
not reasonable to store $U_k$ in dense form. Normally, this is done by
the aforementioned \emph{thin SVD} with $U \in \RR^{d \times r}$,
$\Sigma \in \RR^{r \times r}$ and $V \in \RR^{r \times n}$, where $r$
is the rank of the matrix $A = U\Sigma V^T \in \RR^{d \times n}$. The
storage requirement for such representations is
$d r + r^2 + r n = r ( r+d+n )$ and thus only linear in the dominant
direction $d$.

To obtain a similar storage performance and be able to apply fast
algorithms, we will store $U_k$ as a
factorization by means of \emph{Householder vectors} instead,
cf. \cite{GolubvanLoan96}. Recall that for a vector $y \in
\RR^d$ the Householder reflection matrix $H_y = I - 2 \frac{y y^T}{y^T
  y}$ is a
symmetric orthogonal matrix that can be used for obtaining the QR
factorization. Indeed, since the QR factorization of $Y$ has to
annihilate a lot of numbers simultaneously in any step, it is
reasonable to perform it by Householder reflections so that
\begin{equation}
  \label{eq:HousehQR}
  Q = \prod_{j=1 \uparrow p} ( I - y_j y_j^T ) =: \left( I - y_1 y_1^T \right)
  \cdots \left( I - y_p y_p^T \right), \qquad \| y_j \|_2 = \sqrt{2},
  \quad j=1,\dots,p,
\end{equation}
for some $p \in \NN$. Note that we write the noncommutative matrix
products in a left-to-right way which means that
$$
Q^T = \left( I - y_p y_p^T \right)
\cdots \left( I - y_1 y_1^T \right) =: \prod_{j=p \downarrow 1} ( I - y_j y_j^T
),
$$
which is slight unconventional but convenient.

\begin{lem}\label{L:pnk'}
  If $Q$ is the orthogonal matrix from the $QR$ factorization
  (\ref{eq:YQR}), then $p = r_{k+1} - r_k$.
\end{lem}

\begin{proof}
  The decomposition (\ref{eq:YQRP}) of $Y$ can be written as
  $$
  \widetilde Q^T Y P^T = \left[
    \begin{array}{cc}
      R' & * \\ 0 & 0
    \end{array}
  \right], \qquad R' \in \RR^{p \times p},
  $$
  if the column pivoting terminates after $p$ steps. Hence, by
  (\ref{eq:Rk+1Def}), 
  $$
  R_{k+1} = \left[
    \begin{array}{ccc}
      \Sigma_k' & * & * \\
                & R' & * \\
                & 0 & 0 
    \end{array}
  \right],
  $$
  and therefore $r_{k+1} = \rank R_{k+1} = r_k + p$ since 
  $r_k + p$ diagonal elements of the upper triangular matrix
  $R_{k+1}$ are nonzero.
\end{proof}

\noindent
Since for any orthogonal matrix $U \in \RR^{d \times d}$ one has
$$
H_y U = ( I - y y^T ) U = U - y (U^T y)^T = U \left( I - (U^T y) (U^T
y)^T \right) = U H_{U^T y},
$$
the matrix
$$
Q \left[
  \begin{array}{cc}
    \widetilde U & 0 \\
    0 & I
  \end{array}
\right] =  \left[
  \begin{array}{cc}
    \widetilde U & 0 \\
    0 & I
  \end{array}
\right] \, \prod_{j=1 \uparrow {r_{k+1} - r_k}} ( I - \tilde y_j \tilde y_j^T )
$$
from (\ref{eq:SVDUpdateU}) can be represented by $\widetilde U \in
\RR^{n_{k+1} \times n_{k+1}}$ and the vectors
$$
\tilde y_j := \left[
  \begin{array}{cc}
    \widetilde U^T & 0 \\
    0 & I
  \end{array}
\right] y_j, \qquad j=1,\dots,r_{k+1} - r_k.
$$
Hence, if we assume that we already have computed a representation of
the form
\begin{equation}
  \label{eq:UkRep}
  U_k = \left[
    \begin{array}{cc}
      \widetilde U_k & 0 \\
      0 & I 
    \end{array}
  \right] \prod_{j=1 \uparrow r_k} ( I - h_j h_j^T ),  
\end{equation}
where we store the Householder vectors as columns of a matrix
\begin{equation}
  \label{eq:HousehStorage}
  H_k := \left[ h_1,\dots,h_{r_k} \right] \in \RR^{d \times r_k},
\end{equation}
the update rule (\ref{eq:SVDUpdateU}) becomes
\begin{eqnarray*}
  U_{k+1} & = & U_k Q \left[
    \begin{array}{cc}
      \widetilde U & 0 \\ 0 & I \\
    \end{array}
  \right] \\
  & = & \left[
    \begin{array}{cc}
      \widetilde U_k & 0 \\
      0 & I 
    \end{array}
  \right] \prod_{j=1 \uparrow r_k} ( I - h_j h_j^T )  \,  \left[
    \begin{array}{cc}
      \widetilde U & 0 \\ 0 & I \\
    \end{array}
  \right] \prod_{j=1 \uparrow r_{k+1}- r_k} ( I - \tilde y_j \tilde y_j^T ) \\
  & = & \left[
    \begin{array}{cc}
      \widetilde U_k & 0 \\
      0 & I \\
    \end{array}
  \right]  \left[
    \begin{array}{cc}
      \widetilde U & 0 \\ 0 & I \\
    \end{array}
  \right] \prod_{j=1 \uparrow r_k} ( I - \tilde h_j \tilde h_j^T ) 
  \prod_{j=1 \uparrow r_{k+1} - r_k} ( I - \tilde y_j \tilde y_j^T ) \\
  & = & \left[
    \begin{array}{cc}
      \widetilde U_k & 0 \\
      0 & I \\
    \end{array}
  \right]  \left[
    \begin{array}{cc}
      \widetilde U & 0 \\ 0 & I \\
    \end{array}
  \right] \prod_{j=1}^{r_{k+1}} ( I - \tilde h_j \tilde h_j^T ), 
\end{eqnarray*}
which can be represented by the matrix
\begin{equation}
  \label{eq:Uk+1tilddef}
  \widetilde U_{k+1} =
  \left[
    \begin{array}{cc}
      \widetilde U_k & 0 \\
      0 & I \\
    \end{array}
  \right] \left[
    \begin{array}{cc}
      \widetilde U & 0 \\ 0 & I \\
    \end{array}
  \right]  \in \RR^{n_{k+1} \times n_{k+1}}
\end{equation}
and the $r_{k+1}$ vectors
\begin{equation}
  \label{eq:v}
\tilde h_j = \left[
  \begin{array}{cc}
    \widetilde U^T & 0 \\ 0 & I 
  \end{array}
\right] \,
\left\{
  \begin{array}{ccl}
    h_j, & \quad & j=1,\dots,r_k, \\
    y_{j-r_k}, & & j=r_k+1,\dots,r_{k+1},
  \end{array}
\right.
\end{equation}
or, in matrix notation,
\begin{equation}
  \label{eq:Hmatproduct}
  H_{k+1} = \left[
  \begin{array}{cc}
    \widetilde U^T & 0 \\ 0 & I 
  \end{array}
\right] \, \left[ H_k, \, y_1,\dots,y_{r_{k+1} - r_k} \right].
\end{equation}
The storage requirement for the matrix $H_{k+1}$ on level $k+1$ is
therefore
$n_{k+1} \, ( d + n_{k+1} )$ and the computational effort for the
update step is $O (n_{k+1}^3)$ for the matrix-matrix 
product in (\ref{eq:Uk+1tilddef}) plus $O( n_{k+1}^2 r_{k+1})$ for the
product in
(\ref{eq:v}) since the last $d-n_{k+1}$ entries in each
column of the product can simply be copied. Thus, we can estimate the
computational effort by a total of $O \left( n_{k+1}^2 ( n_{k+1} +
  r_{k+1} ) \right)$.

\begin{rem}
  Another advantage of storing the matrix $U_k$ in Householder
  factorized form is the fact that it is automatically
  orthogonal. Especially when the rank remains relatively stable over
  many iterations, the ``loss of orthogonality'' described in
  \cite{brand06:_fast} will not occur so easily.
\end{rem}

For the computation of $Y$ we note that, with a proper row
partitioning of $B_k$, we get
\begin{eqnarray}
  \nonumber
  \left[
  \begin{array}{c}
    Y' \\ Y \\
  \end{array}
  \right]
& = & U_k^T B_k
      = \prod_{j=r_k \downarrow 1} (I - h_{j} h_{j}^T ) \,
      \left[
      \begin{array}{cc}
        \widetilde U_k^T & 0 \\
        0 & I
      \end{array}
            \right] \left[
            \begin{array}{c}
              B_{k,1} \\ B_{k,2}
            \end{array}
  \right] \\
  \label{eq:ZComput}
& = & \prod_{j=r_k \downarrow 1} (I - h_{j} h_{j}^T ) \, \left[
      \begin{array}{cc}
        \widetilde U_k^T B_{k,1} \\ B_{k,2}
      \end{array}
  \right].
\end{eqnarray}
To initialize the procedure for a column vector interpreted as a matrix
$A_1 \in \RR^{d \times 1}$, we determine $h_1 \in \RR^d$ such that $(I
- h_1 h_1^T) A_1 = \| A_1 \|_F \, e_1$, hence
\begin{equation}
  \label{eq:y1Def}
  A_1 = ( I - h_1 h_1^T ) \left[
    \begin{array}{c}
      \| A_1 \|_F \\
      0 
    \end{array}
  \right], \qquad \| h_1 \|_2 = \sqrt2,
\end{equation}
which is a valid SVD of $A_1$ with $V = 1$ and yields the initialization
\begin{equation}
  \label{eq:UDecompInit}
  \widetilde U_1 = [], \qquad H_1 = h_1
\end{equation}

\noindent
We summarize the procedure in Algorithm~\ref{A:Augment}, which computes
the SVDs of a series of augmented matrices.

\begin{algorithm}
  \caption{Augmented SVD}
  \label{A:Augment}
  \begin{algorithmic}[1]
  \STATE{\textbf{Given:} Matrices $B_j \in \RR^{d \times m_j}$, $j \in \NN$,
    and $A_1 \in \RR^{d \times 1}$.}
  \STATE\label{it:AAugment1}{(Initialization) Determine the Householder
    vector $h_1$ such that (\ref{eq:y1Def}) is satisfied and set
    $$
    \widetilde U_1 = [], \, \widetilde V_1 = [ 1 ],
    \qquad n_k = n_k' = 1, \qquad \Sigma_k' = [ \| A_1 \|_F ] \in \RR^{1 \times 1}
    $$
    as well as $H_1 = h_1$.}
  \FOR{$k=1,2,\dots$}\label{it:AAugment2}
  \STATE\label{it:AAugment2.1}{Compute $Z = U_k^T B_k$ according to
    (\ref{eq:ZComput}) by applying Householder reflections with the
    columns of $H_k$ in reverse order to the matrix $\left[
      \begin{array}{c}
        \widetilde U_k^T B_{k,1} \\ B_{k,2} 
      \end{array}
    \right]$.}
    \STATE\label{it:AAugment2.2}{Compute the QR decomposition with
      column pivoting: 
      $$
      Z_{r_k+1:d,:} = QRP, \qquad
      Q = \prod_{j=1 \uparrow p} ( I - y_j y_j^T ) \in \RR^{d-r_k \times d-r_k},
      $$
      and $P \in \RR^{m_k \times m_k}$ is a permutation. Set $h_j :=
      \left[
        \begin{array}{c}
          0_{r_k} \\ y_j 
        \end{array}
      \right]$, $j=1,\dots,p$.}
    \STATE\label{it:AAugment2.3}{Compute, by means of
      Algorithm~\ref{A:thresholdSVD}, a thresholded singular value
      decomposition of the upper triangular matrix
      $$
      \left[
        \begin{array}{cc}
          \Sigma_k' & Z_{1:n_k',:} P^T \\ & R 
        \end{array}
      \right] = U \Sigma V^T.
      $$}
    \STATE\label{it:AAugment2.4} {Define $r_{k+1} := \max \{ j :
      \sigma_j > 0 \}$ and $\Sigma_{k+1}' = \Sigma_{1:r_{k+1},1:r_{k+1}}$.}
    \STATE\label{it:AAugment2.5}{(Update) Set
      \begin{eqnarray*}
        \widetilde U_{k+1} & = & \left[
                                 \begin{array}{cc}
                                   \widetilde U_k & 0 \\ 0 & I
                                 \end{array}
                                                             \right] \, U, \\
        V_{k+1}
                           & = & \left[
                                 \begin{array}{cc}
                                   V_k & 0 \\ 0 & I
                                 \end{array}
                                                  \right] \left[
                                                  \begin{array}{ccc}
                                                    I_{n_k'} & 0 & 0 \\
                                                    0 & 0 & I_{n_k - n_k'} \\
                                                    0 & P^T & 0
                                                  \end{array}
                                                              \right] \left[
                                    \begin{array}{cc}
                                      V & 0 \\ 0 & I
                                    \end{array}
                                                   \right], \\
        H_{k+1} & = & \left[
                      \begin{array}{cc}
                        U^T & 0 \\ 0 & I
                      \end{array}
          \right] \left[ H_k, \, h_1,\dots,h_p \right].
      \end{eqnarray*}}
    \ENDFOR    
  \end{algorithmic}
\end{algorithm}

\begin{lem}\label{T:AlgEffort}
  In each step, Algorithm~\ref{A:Augment} computes an approximate SVD
  of $A_k$, and a precise SVD if $\tau = 0$, where the $k$th step
  requires 
  \begin{equation}
    \label{eq:TAlgEffort}
    O \left( n_{k+1}^3 + d m_k ( r_k + m_k ) \right)  
  \end{equation}
  floating point operations and the memory consumption for $A_k$ is bounded by
  $O ( n_k^2 + r_k d )$.
\end{lem}

\begin{proof}
  The validity of the algorithm follows from the preceding exposition
  where the individual steps have been introduced. Let us count the
  computational effort in the individual steps of the iteration in
  \ref{it:AAugment2}).  According to (\ref{eq:ZComput}) we first
  compute in \ref{it:AAugment2.1}) a product of an $n_k \times n_k$
  and an $n_k \times m_k$ matrix, while retaining $B_{k,2}$, which
  needs $O ( n_k^2 m_k )$ operations, and then $r_k$ Householder
  reflections on a $d \times m_k$ matrix, which contributes $O (r_k
  d m_k)$ flops,
  hence the total effort is $O \left( m_k ( n_k^2 + d r_k ) \right)$
  flops. According to \cite[Algorithm~5.2.1]{GolubvanLoan96},
  Householder $QR$ of the $(d-r_k )\times m_k$ matrix $Z_{r_k+1:d,:}$
  in \ref{it:AAugment2.2})
  needs $O \left( m_k^2 (d-r_k) \right)$ operations. In
  \ref{it:AAugment2.3}), the computation of $Z_{1:r_k,:} P^T$ requires
  $O \left( r_k m_k^2 \right)$ 
  flops while the SVD itself, as SVD of an upper triangular matrix, can
  be done in $O \left((r_k+m_k)^3 \right)$ operations, see
  Remark~\ref{R:SVDEffort}. Since $\widetilde U_k$ and $V_k$ are $n_k
  \times n_k$ matrices, the effort for the first two updates in step
  \ref{it:AAugment2.5}) is another $O \left( (n_k + m_k)^3 \right) =
  O \left( n_{k+1}^3 \right)$, while the update of $H_{k+1}$ can be done with at
  most $O \left( n_{k+1}^2 ( n_{k+1} + r_{k+1} ) \right)^2 = O \left(
    n_{k+1}^3 \right)$
  operations. With the obvious estimate $r_k \le n_{k+1}$ and $m_k \le
  n_{k+1}$, we can sum up everything to give
  (\ref{eq:TAlgEffort}). The memory effort is clear since we only
  store the $n_k \times n_k$ matrices $\widetilde U_k$ and $V_k$ and
  the Householder vectors $H_k$ as a $d \times r_k$ matrix,
  see Lemma~\ref{L:pnk'}.
\end{proof}

\begin{rem}
  The main advantage of Algorithm~\ref{A:Augment} is that its effort
  in computation and memory depends only \emph{linearly} on the
  column size $d$ which makes tailored for problems where small blocks
  of large columns are added to a matrix.
\end{rem}

\begin{rem}
  The complexity of our algorithm is comparable to that of the method
  proposed by Brand in \cite{brand06:_fast} who reports $O (r_k^3 + d
  r_k)$ for an update by a single column, i.e., $m_k = 1$. The
  slightly higher $n_{k+1}$ in \eqref{eq:TAlgEffort} is reflecting the
  fact that we always compute the \emph{full} matrix $V_k \in \RR^{n_k
    \times n_k}$ since it immediately gives a basis for the kernel of
  the matrix. Note that in the special case of appending a single
  column to a thin SVD of a full rank matrix, our estimate coincides
  with the one from \cite{brand06:_fast}, but is slightly better in
  the case of appending several columns if the rank is increased
  during this process.
\end{rem}

\section{Thresholding}
\label{sec:thresh}
Now we attack the problem of choosing a proper threshold level for the
upper triangular matrix $\widetilde R$ in \eqref{eq:YQRP}. To that end
we assume that a square upper triangular matrix $R \in \RR^{n \times
  n}$ can be partitioned as
\begin{equation}
  \label{eq:RprPart}
  R = \left[
    \begin{array}{cc}
      R_{11} & R_{21} \\ 0 & R_{22}
    \end{array}
  \right], \qquad R_{11} \in \RR^{m \times m}, \quad R_{22} \in
  \RR^{(n-m) \times (n-m)}.
\end{equation}
with
\begin{equation}
  \label{eq:rDecrease}
  |r_{11}| \ge \cdots \ge |r_{mm}| \ge |r_{m+1,m+1}| \ge \cdots
  \ge |r_{nn}|,
\end{equation}
and
\begin{equation}
  \label{eq:rRowDecrease}
  | r_{jj} | \ge | r_{j,j+1} | \ge \cdots \ge | r_{jn} |,
\end{equation}
which is guaranteed in the preceding section by computing a QR
decomposition with column pivoting.

Given a threshold $\tau > 0$, we want to use information on $R$ to
threshold $R$ in such a way that only singular values of $R$ with
$\sigma \le \tau$ are set to zero and that as many of the singular
values $> \tau$ as possible are preserved.

To that end, let $R = U \Sigma V^T$ denote the singular value
decomposition with $\sigma_1 \ge \cdots \ge \sigma_n \ge 0$.
As mentioned in \cite{chandrasekaran94}, the interlacing property of
singular values readily implies that
\begin{equation}
  \label{eq:RSigmaEst}
  \sigma_m ( R_{11} ) \le \sigma_m \qquad \sigma_1 ( R_{22} ) \ge
  \sigma_{m+1},   
\end{equation}
so a good separation between the singular values is obtained if $\tau$
is chosen such that $\sigma_m ( R_{11} )$ is large while $\sigma_1
( R_{22} )$ is small. To that end, we first show that the threshold
carries over up to a quantity that is linear in the number of
thresholded diagonal values.

\begin{lem}\label{L:Rsigk+1Est}
  For any given $m$ the partition (\ref{eq:RprPart}) satisfies
  \begin{equation}
    \label{eq:Rsigk+1Est}
    \sigma_{m+1} < \sqrt{\frac{(n-m)(n-m+1)}{2}} \, r_{m+1,m+1}.
  \end{equation}
\end{lem}

\begin{proof}
  Since $R_{22}$ contains $\frac{(n-m)(n-m+1)}{2}$ nonzero elements of modulus
  $\le r_{m+1,m+1}$, we find that
  $$
  \sigma_1^2 (R_{22}) \le 
  \sum_{j=1}^{n-m} \sigma_j^2 (R_{22}) = \| R_{22} \|_F \le r_{m+1,m+1}^2 \,
  \frac{(n-m)(n-m+1)}{2} 
  $$
  due to (\ref{eq:rDecrease}) and (\ref{eq:rRowDecrease}),
  and (\ref{eq:Rsigk+1Est}) follows directly from
  (\ref{eq:RSigmaEst}). 
\end{proof}

\noindent
Therefore, if we choose the index $m$ as
\begin{equation}
  \label{eq:ThresholdChoice}
  m = \min \left\{ j : | r_{j+1,j+1} | \le
    \sqrt{\frac{2}{(n-j)(n-j+1)}} \, \tau \right\}
\end{equation}
and pass the matrix
\begin{equation}
  \label{eq:ThreshatMatDef}
  \hat R = \left[
  \begin{array}{cc}
    R_{11} & * \\ 0 & 0
  \end{array}
\right]  
\end{equation}
to the SVD computation, the above reasoning then shows that
$\tau \ge \sigma_{m+1} (R) \ge \cdots \ge
\sigma_n (R)$ while, by thresholding construction, $\sigma_{m+1} (\hat
R) = \cdots = \sigma_n ( \hat R) = 0$. In other words, the thresholding applied
to $R$ only transforms singular values to zero that fall below the
prescribed threshold level.

A reasonable lower estimate for
$\sigma_m$ based on a lower bound on $|r_{mm}|$ alone is impossible, as
the well--known matrix
$$
\left[
  \begin{array}{cccc}
  1 & -1 & \dots & -1 \\
  & \ddots & \ddots & \dots \\
  & & \ddots & -1 \\
  & & & 1
  \end{array}
\right] \in \RR^{n \times n}
$$
shows, whose smallest singular value decays exponentially in the matrix
dimension $n$ but all of whose diagonal elements are $1$.

A checkable and even computable bound for $\sigma _k$ is the following probably
well--known fact that we prove for the sake of completeness.

\begin{lem}\label{L:sigkR11lower}
  Let $R_{11} = D ( I - N )$ where $D = \diag \left( r_{jj} : j
    =1,\dots,m \right)$ and $N$ is a nilpotent upper triangular matrix.
 Then
  \begin{equation}
    \label{eq:sigkR11lower}
    \sigma_m (R_{11}) \ge r_{mm} \, \left\| \sum_{j=0}^{m-1} N^j
    \right\|_F^{-1}.    
  \end{equation}
\end{lem}

\begin{proof}
  For any $x \neq 0$ we have
  $$
  \| x \|_2 = \| R_{11}^{-1} R_{11} x \|_2 \le \| R_{11}^{-1} \|_2 \|R_{11} x \|_2,
  $$
  hence
  $$
  \frac{\|R_{11} x \|_2}{\| x \|_2} \ge \| R_{11}^{-1} \|_2^{-1} \ge  \| R_{11}^{-1}
  \|_F^{-1}, \qquad x \neq 0,
  $$
  which also holds for the minimum of this expression, which is the
  smallest singular value $\sigma_k (R_{11})$. Using the decomposition
  $R_{11} = D (I - N)$ we then find that
  $$
  \| R_{11}^{-1} \|_F \le r_{mm}^{-1} \, \left\| \sum_{j=0}^{m-1} N^j
  \right\|_F,
  $$
  which is (\ref{eq:sigkR11lower}).
\end{proof}

\noindent
The definition of $m$ in (\ref{eq:ThresholdChoice}) then yields that
$$
\sqrt{\frac{2}{(n-m)(n-m-1)}} \, \tau < | r_{mm} |,
$$
hence
\begin{equation}
  \label{eq:sigmakLower}
  \sigma_m (R) \ge \sigma_m (R_{11}) > \sqrt{\frac{2}{(n-m)(n-m-1)}}
  \left\| \sum_{j=0}^{m-1} N^j \right\|_F^{-1} \, \tau.
\end{equation}
This estimate explains how the conditioning of $R_{11}$ affects the
leading $k$ singular values of $R$. In particular, it can happen that
the SVD detects further almost kernel elements of $R$ that are not
found by the QR decomposition, which is another reason to prefer the
SVD to the simpler rank revealing factorizations.

It has to be mentioned that there are improved
pivoting strategies, described in \cite{chandrasekaran94}, but since
most of them require the computation of an SVD as an auxiliary tool, it
is more efficient to stick with the SVD. Note, however, that clearly
$$
\sigma_1 \ge \sigma_1 (R_{11} ) \ge \tau
$$
and that the interlacing property of singular values,
cf. \cite{GolubvanLoan96},  yields that, after thresholding, the
thresholded matrix $\hat R$ from \eqref{eq:ThreshatMatDef} satisfies
$$
\sigma_m ( \hat R ) \ge \sigma_k ( R_{11} ) > \sigma_{m+1} (\hat R) = \cdots
= \sigma_n ( \hat R) = 0.
$$
The Wielandt--Hoffman theorem for singular values,
\cite[Theorem~8.6.4]{GolubvanLoan96}, shows that we get a
reasonable approximation for the singular values of $\hat R=R+E$ for
some $E\in \RR^{n \times n}$:
\begin{equation}
  \label{eq:WielandtHoffman}
  \sum_{j=1}^n \left( \sigma_j (R) - \sigma_j (\hat R) \right)^2 \le \|
  R-\hat R \|_F^2 < \tau^2,  
\end{equation}
which immediately gives following result.

\begin{lem}\label{P:RankRobust}
  If
  $$
  \sigma_k ( R_{11} ) \ge \tau + \sigma
  $$
  then $\sigma_1 (\hat R) \ge \cdots \ge \sigma_k (\hat R) > \sigma$,
  that is, the rank of $\hat R$ is observed correctly relative to the
  threshold $\tau$.
\end{lem}

\begin{proof}
  Since $\sigma_k ( \hat R ) \ge \sigma_k ( R_{11} )$ and, by
  (\ref{eq:WielandtHoffman}),
  $$
  \left| \sigma_k ( \hat R ) - \sigma_k (R) \right| \le \left(
    \sum_{j=1}^n \left(\sigma_j (R) - \sigma_j ( \hat R ) \right)^2
  \right)^{1/2} \le \tau,
  $$
  we get that $\sigma_k (R') > \sigma$ as claimed.
\end{proof}

\noindent
The structure of the matrix
$$
R_{k+1} = \left[
  \begin{array}{cc}
    \Sigma_k' & * \\
    0 & \widetilde R
  \end{array}
\right]
$$
from (\ref{eq:Rk+1Def}) allows us to draw further conclusions on the
singular values of $R_{k+1}$ together with the inductive assumption
that $\Sigma_k'$ results from a thresholding process with threshold
level $\tau$ yielding that $( \Sigma_k'
)_{jj} \ge \tau$. Adding one column to $\Sigma_k'$ obtaining the
matrix
$$
S = \left[
  \begin{array}{cc}
    \Sigma_k' & * \\ 0 & *
  \end{array}
\right],
$$
the interlacing property of singular values yields $\sigma_{n_k'} (S)
\ge \sigma_{n_k'} ( \Sigma_k' ) \ge \sigma_{n_{k+1}'} (S)$, hence
$\sigma_k (S) \ge \tau$. By an inductive repetition of this argument
it follows that
$$
\sigma_{n_k'} ( R_{k+1} ) \ge \sigma_{n_k'} ( \Sigma_k' ) \ge \tau.
$$
This reasoning remains unchanged if we decompose $R_{k+1}$ according
to the thresholding strategy (\ref{eq:ThresholdChoice}) into
\begin{equation}
  \label{eq:Rk+1TriDecomp}
  R_{k+1} = \left[
    \begin{array}{cc}
      R_{11} & R_{12} \\ 0 & R_{22}
    \end{array}
  \right], \qquad R_{11} \in \RR^{m \times m}.
\end{equation}
Then $m \ge n_k'$ since all diagonals of $\Sigma_k'$ exceed
$\tau$. Since the above reasoning depends only on adding columns to
$\Sigma_k'$, we can draw the following conclusion.

\begin{lem}\label{P:ThreshNumerRank}
  The thresholded matrix
  $$
  \hat R_{k+1} := \left[
    \begin{array}{cc}
      R_{11} & R_{12} \\
      0 & 0 \\
    \end{array}
  \right]
  $$
  satisfies $\sigma_{n_k'} \left( \hat R_{k+1} \right) \ge \tau$.
\end{lem}

\noindent
With this information at hand, we can fix our pivoting structure to
compute the matrix $\Sigma_{k+1}'$ in Algorithm~\ref{A:thresholdSVD}.

\begin{algorithm}
  \caption{SVD Thresholding}
  \label{A:thresholdSVD}
  \begin{algorithmic}[1]
  \STATE{\textbf{Given:} matrix $R_{k+1}$ of the form $\left[
    \begin{array}{cc}
      \Sigma & * \\ 0 & R
    \end{array}
  \right]$.}
  \STATE{Decompose $R_{k+1}$ according to
    (\ref{eq:Rk+1TriDecomp}) with the thresholding strategy
    (\ref{eq:ThresholdChoice}).}
  \STATE{Set $R_{22} = 0$.}
  \STATE{Compute the SVD
    $$
    R_{k+1}' = U \left[
      \begin{array}{cc}
        \Sigma & 0 \\
        0 & 0      
      \end{array}
    \right]
    V^T, \qquad \Sigma \in \RR^{m \times m},
    $$
    and truncate $\Sigma$ with threshold $\tau$. In other words, write
    $$
    \Sigma = \left[
      \begin{array}{cc}
        \Sigma_{k+1}' & 0 \\
        0 & D
      \end{array}
    \right],
    \qquad \Sigma_{k+1}' \in \RR^{n_{n+1}' \times n_{k+1}'},
    $$
    such that all diagonal elements of $\Sigma_{k+1}'$ are $\ge \tau$
    and all diagonal elements of $D$ are $< \tau$.}
  \end{algorithmic}
\end{algorithm}

\noindent
Due to the above arguments this strategy has a very important
property.

\begin{lem}\label{T:ThreshRank}
  The thresholding strategy is rank increasing, i.e., $r_{k+1} \ge r_k$.
\end{lem}

\begin{rem}\label{R:SVDEffort}
  Computing the SVD of the upper triangular matrix can be done in $O \left(
    (r_k+m_k)^3 \right)$ operations, see the comments on the R--SVD in
  \cite[Chapter~5.4]{GolubvanLoan96}.
\end{rem}

\section{Applications and numerical experiments}
\label{sec:applications}
To motivate and justify the development of the methods in the
preceding sections, we finally point out two main applications where
they turn out to be useful.
Algorithm~\ref{A:Augment} has been implemented prototypically in
\texttt{octave}
\cite{eaton09:_gnu_octav}. The code can be downloaded for checking and
verification from
\begin{center}
  \small
  \verb!www.fim.uni-passau.de/digitale-bildverarbeitung/forschung/!
  \verb!downloads!  
\end{center}
All tests and experiments in the following section refer to this software.

\subsection{Absolute and relative thresholding}
\label{sec:Thresh}

Before we describe a simple experiment and the applications, we must
make clear that the thresholding and thus rank detection strategy
indicated by Lemma~\ref{T:ThreshRank} is not the usual numerical rank
detection strategy as used, for example, by the \texttt{rank} command
in \texttt{octave}. There a singular value $\sigma$ of $A \in \RR^{m
  \times n}$, $m\ge n$, is thresholded to zero if $\sigma \le  m \sigma_1 u$,
where $u$ is the \emph{unit roundoff} that describes the numerical
accuracy, cf. \cite{Higham02}, and $\sigma_1 = \| A \|_2$ is the
largest singular value of $A$. Though this strategy is the only reliable
general purpose rank detection one, especially since it is
independent of normalization, there is a phenomenon that particularly
affects
the two applications below: the more the matrix grows, the larger the
threshold level becomes and more and more singular values will be
thresholded to zero. A direct application of the \texttt{rank} command in
the Prony algorithm of \cite{Sauer17_Prony} even gave
\emph{decreasing} ranks for augmented matrices sometimes. Another
simple example would be video analysis: imagine that a still image,
even a normalized one, is transmitted over a fairly long period, say
$N$ times. Then
the respective singular value will be $\sqrt{N}$, where $N$ denotes
the number of repetitions, hence the threshold level will grow at least like
$N^{3/2}$ and may become so large that standard rank
methods will ignore almost any frame, even if it is significantly
different. On the other, the computation of the matrix $R$ in
Algorithm~\ref{A:thresholdSVD} depends on the singular vectors only,
not on the singular \emph{values} accumulated so far, hence the roundoff
errors affecting this matrix would still be independent of $N$ and an
absolute threshold will reliably detect the difference, in contrast to
a relative one.
Taking into account that, by the Wielandt--Hoffman theorem
for singular values, cf. \cite[Theorem~8.6.4]{GolubvanLoan96},
$$
\sum_{j=1}^r \left( \sigma_j (A') - \sigma_j (A) \right)^2 \le \| A -
A' \|_F^2
$$
perturbations of a matrix only affect the singular values in an
additive way, it makes sense, especially in the applications below, to
use methods that work with an \emph{absolute} thresholding that
sets singular values to zero if they fall below a certain absolute
value $\tau$. This is the threshold strategy developed in the
preceding section.




\subsection{Video analysis}
\label{sec:vide}
The first application is the computation
of \emph{principal components} for sequential data. Principal
Component Analysis is a classical and frequently used technique in
signal analysis and (unsupervised) machine learning,
cf. \cite{hastie09:_elemen_statis_learn}, and essentially consists of
finding the best low dimensional approximation to \emph{feature
  vectors} $y_k \in \RR^d$, $k=1,2,\dots,n$. Arranging these features
into a matrix $Y \in \RR^{d \times n}$, the best low dimensional
approximation with respect to the Euclidean norm corresponds to
finding a matrix $X \in \RR^{d \times n}$ of rank, say $n' < n$, such
that the \emph{Frobenius norm} $\| Y - X \|_F$ is minimized. This
matrix, on the other hand, is obtained by choosing the first $n'$
columns of $U$ in the SVD $Y = U \Sigma V^T$. Especially in imaging
applications, where the features can be the pixel values, color or
greyscale, of an image, $d$ can be large and handling or processing
the full matrix $Y$ is difficult. Moreover, the features may not all
be present at the beginning and storing them first will also cause
complications.

Moreover, since the noise level in video images exceeds the unit
roundoff error by orders of magnitude, a relatively high absolute
threshold $\tau$ in the SVD update is possible and also adds a
denoising effect to the computations.

As an example application, we consider PCA analysis to detect moving
objects in videos of fixed view cameras. To that end, we consider
a set of 594
greyscale images from a webcam viewing the city hall of Passau  on
February 7, 2016. These images are handled in full $640 \times 480$
resolution, yielding a matrix $A \in \RR^{307200 \times 594}$ of rank
$578$ which were read in chunks of $30$ images. The \texttt{octave}
implementation of the algorithm worked out of the box with a rate of
about 7 frames per second.
From the Householder representation, the $p$th
singular vectors can be easily and efficiently computed as
$$
\left[
  \begin{array}{cc}
    \widetilde U & 0 \\
    0 & I 
  \end{array}
\right] \prod_{j=1 \uparrow r_k} ( I - h_j h_j^T ) \left[
  \begin{array}{c}
    e_p \\ 0
  \end{array}
\right], \qquad p=1,\dots,r_k.
$$
As a simple example, Fig.~\ref{fig:SingVecs} shows the
dominant singular vector and the fairly irrelevant $473$rd singular
vector of the video sequence.
\begin{figure}
  \centering
  \includegraphics[width=.46\hsize]{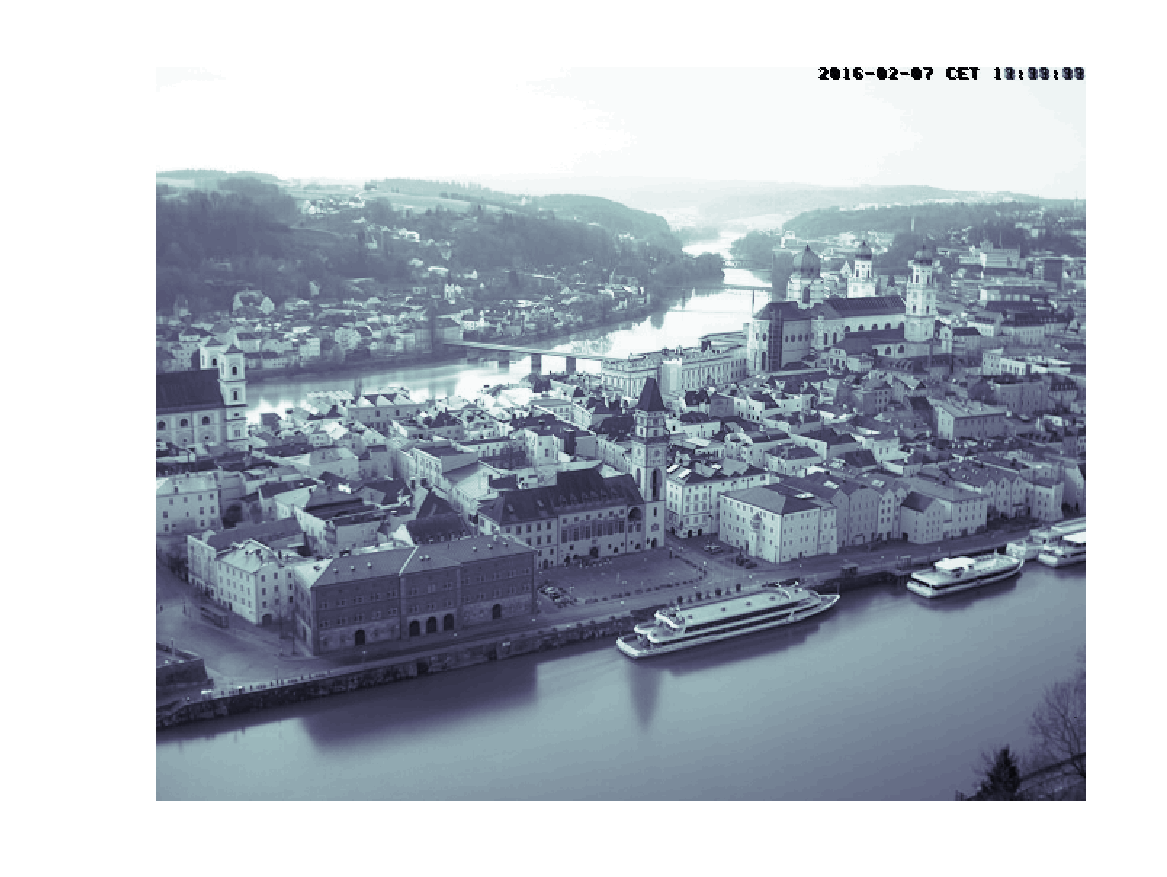}\qquad
  \includegraphics[width=.46\hsize]{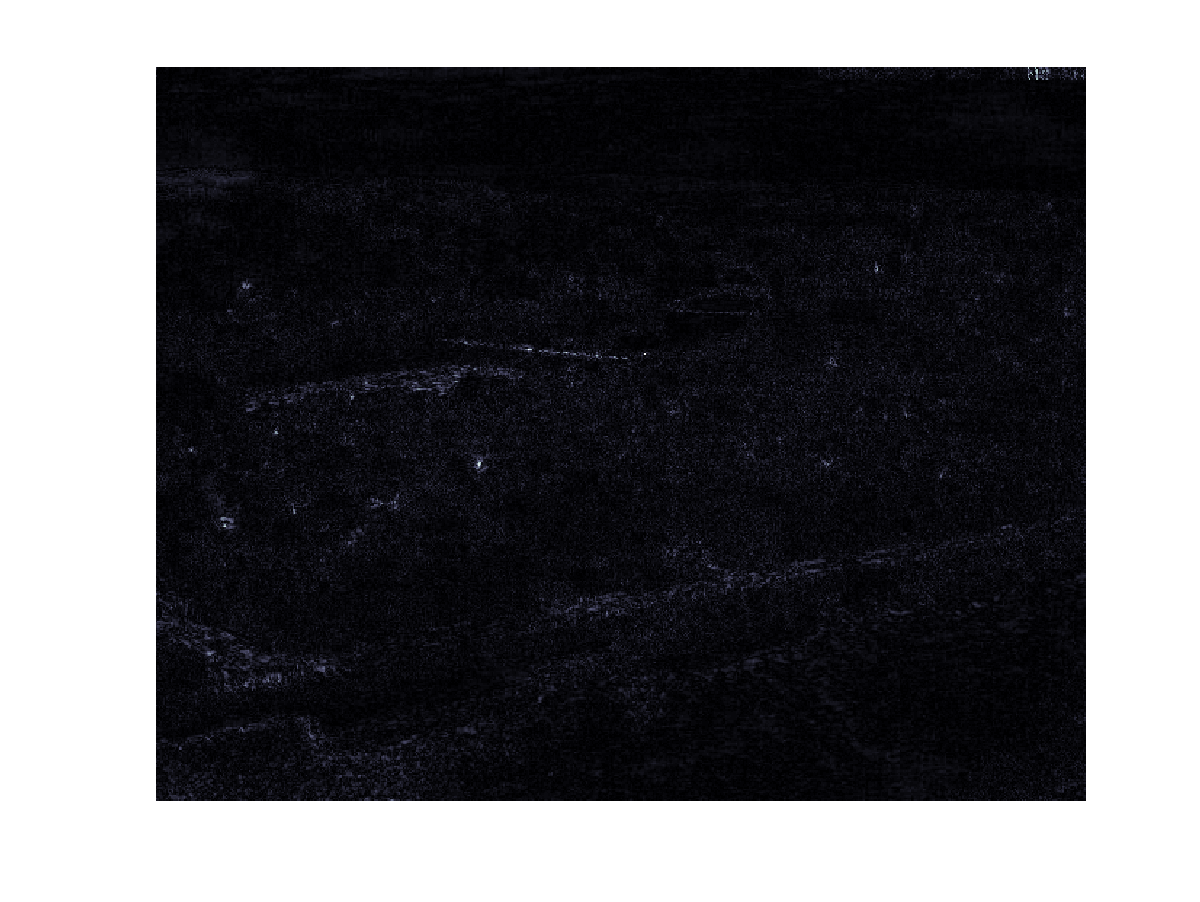}
  \caption{First and $473$rd singular vector of the image sequence}
  \label{fig:SingVecs}
\end{figure}

\begin{figure}
  \centering
  \includegraphics[width=.46\hsize]{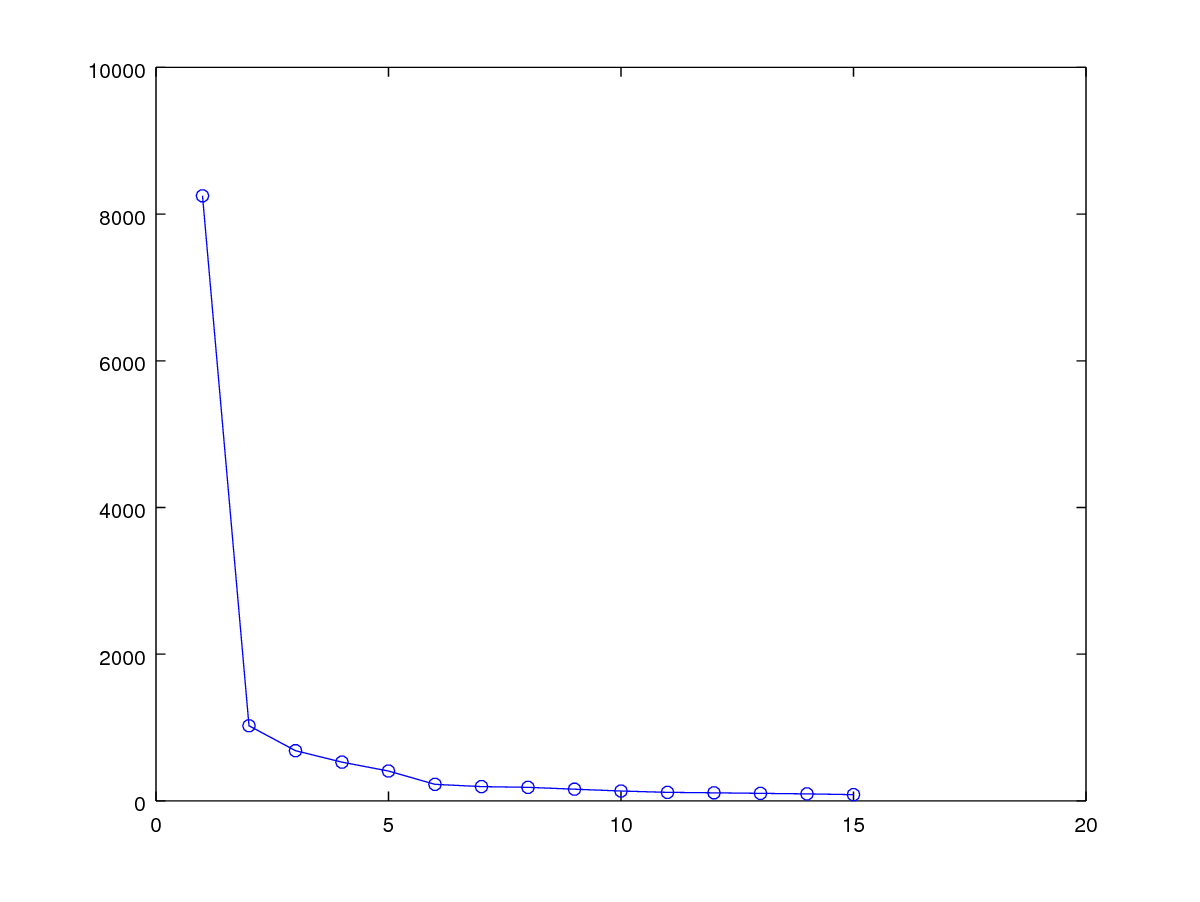}\qquad
  \includegraphics[width=.46\hsize]{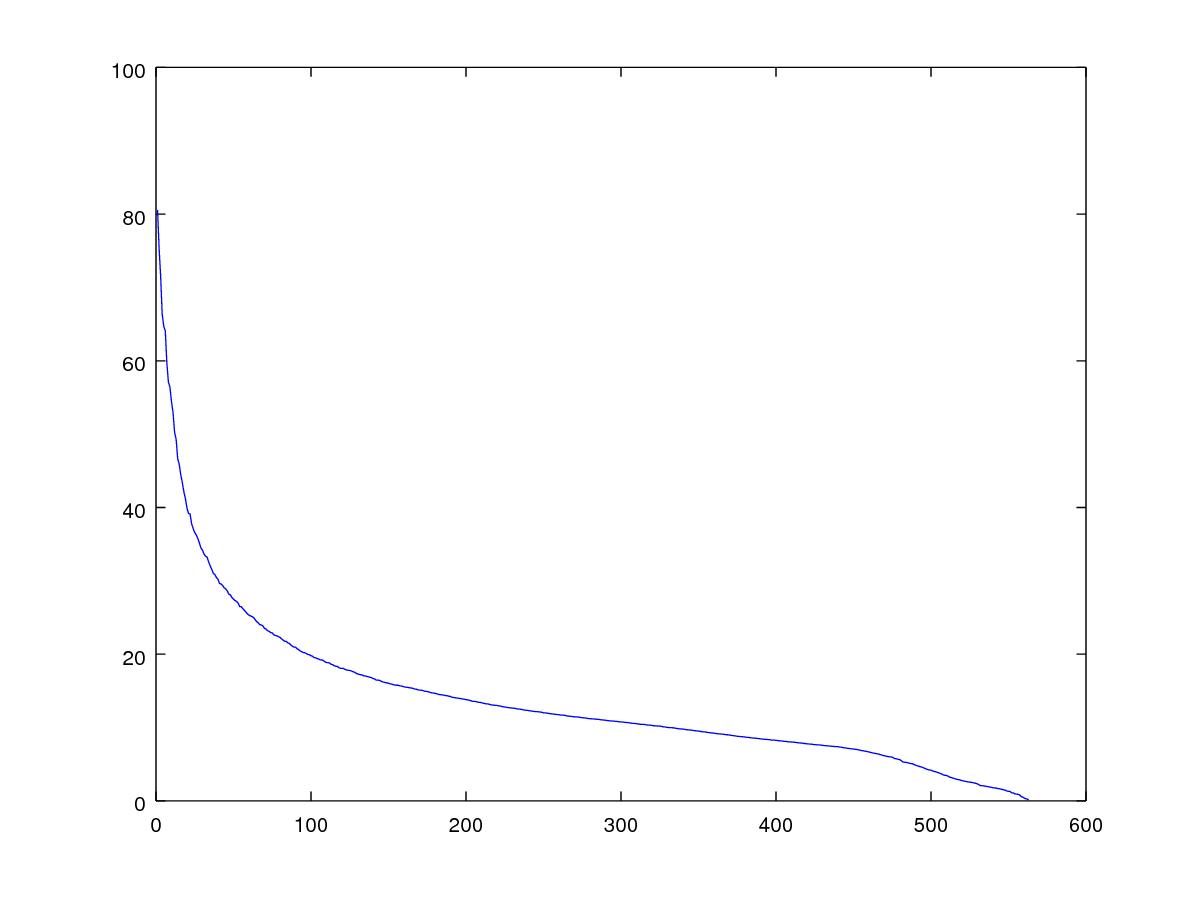}
  \caption{The first $15$ singular values \emph{(left)} of the image
    and the remaining ones \emph{(right)}. Note the scaling of the $y$--axis.}
  \label{fig:SingVals}
\end{figure}

Indeed, the singular values decay rapidly. This can be seen in
Fig.~\ref{fig:SingVals}, where we have decomposed the singular values
into the dominant 15 and the remaining ones. The L-shape of this curve
suggests to cut down to about $20$ singular values and to decompose
the sequence by projecting on the
first $20$ and on the remaining singular vectors. This essentially
removes moving objects from the frames but still maintains more
persistent features like shadows and illumination of the scenery which
change over time, but in a slower and more persistent way. We
show two example frames in Fig.~\ref{fig:BoatBus} and
Fig.~\ref{fig:Ripples}, where the top left image is the original frame
that is decomposed into a ``still'' image and an image with the
``moving'' parts.  Note that the advantage of our algorithm is that,
in contrast to methods like
\cite{rodriguez16:_increm_princ_compon_pursuit_video_backg_model}
which is based on \cite{baker12:_low_rank}, it
allows to compute the projection of the frames to an arbitrary number
of singular vectors, once the video is learned properly. Note that the
number of \emph{relevant} singular values can usually only be detected
once the SVD is computed.

The original frames and video with the decomposition can also be
downloaded for verification from the address given above.

\begin{figure}
  \centering
  \includegraphics[width=.93\hsize]{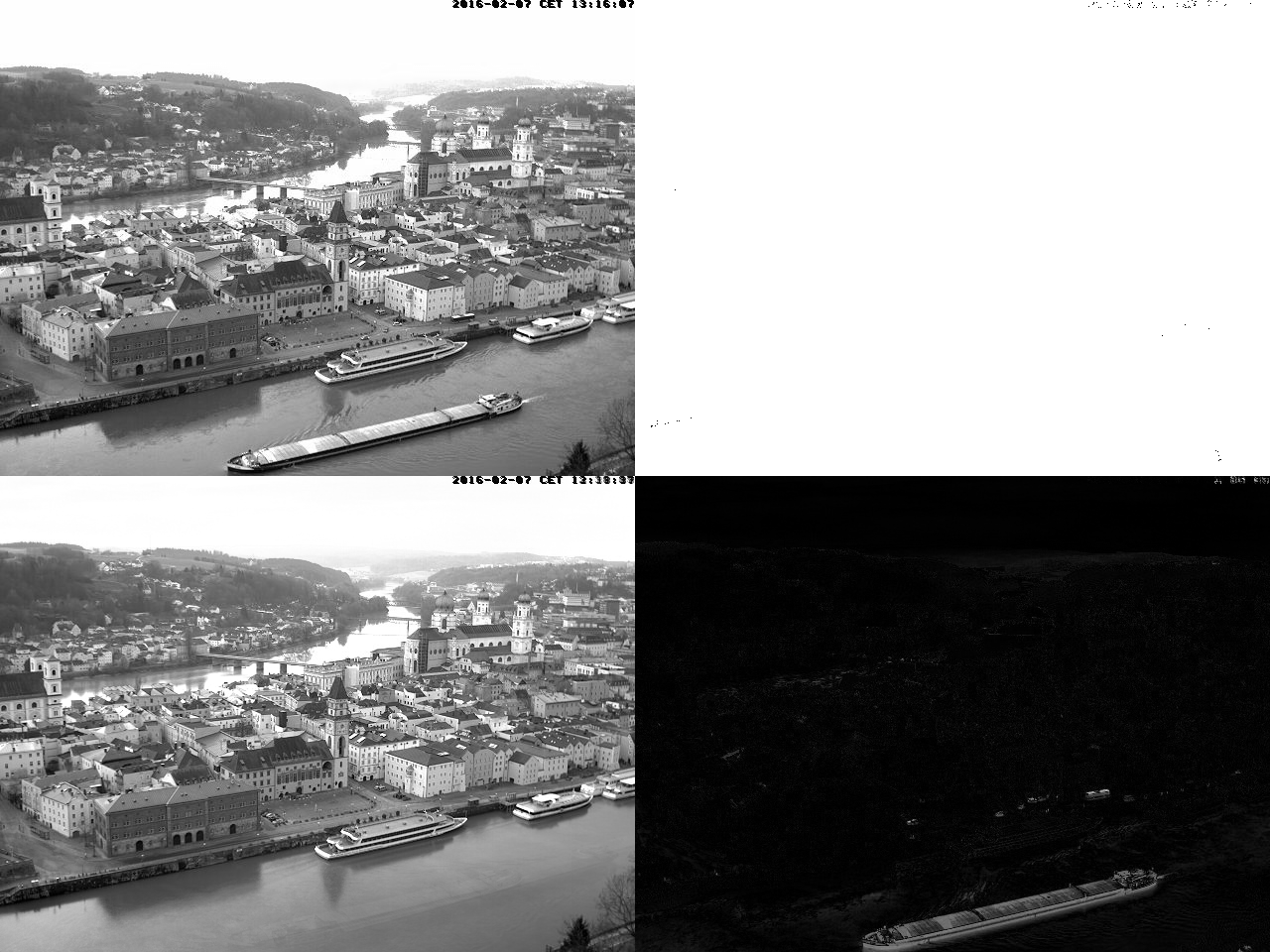}
  \caption{Frame \#351, decomposition: Detection of large (boat) and
    small (bus) moving features at the same time.}
  \label{fig:BoatBus}
\end{figure}

\begin{figure}
  \centering
  \includegraphics[width=.93\hsize]{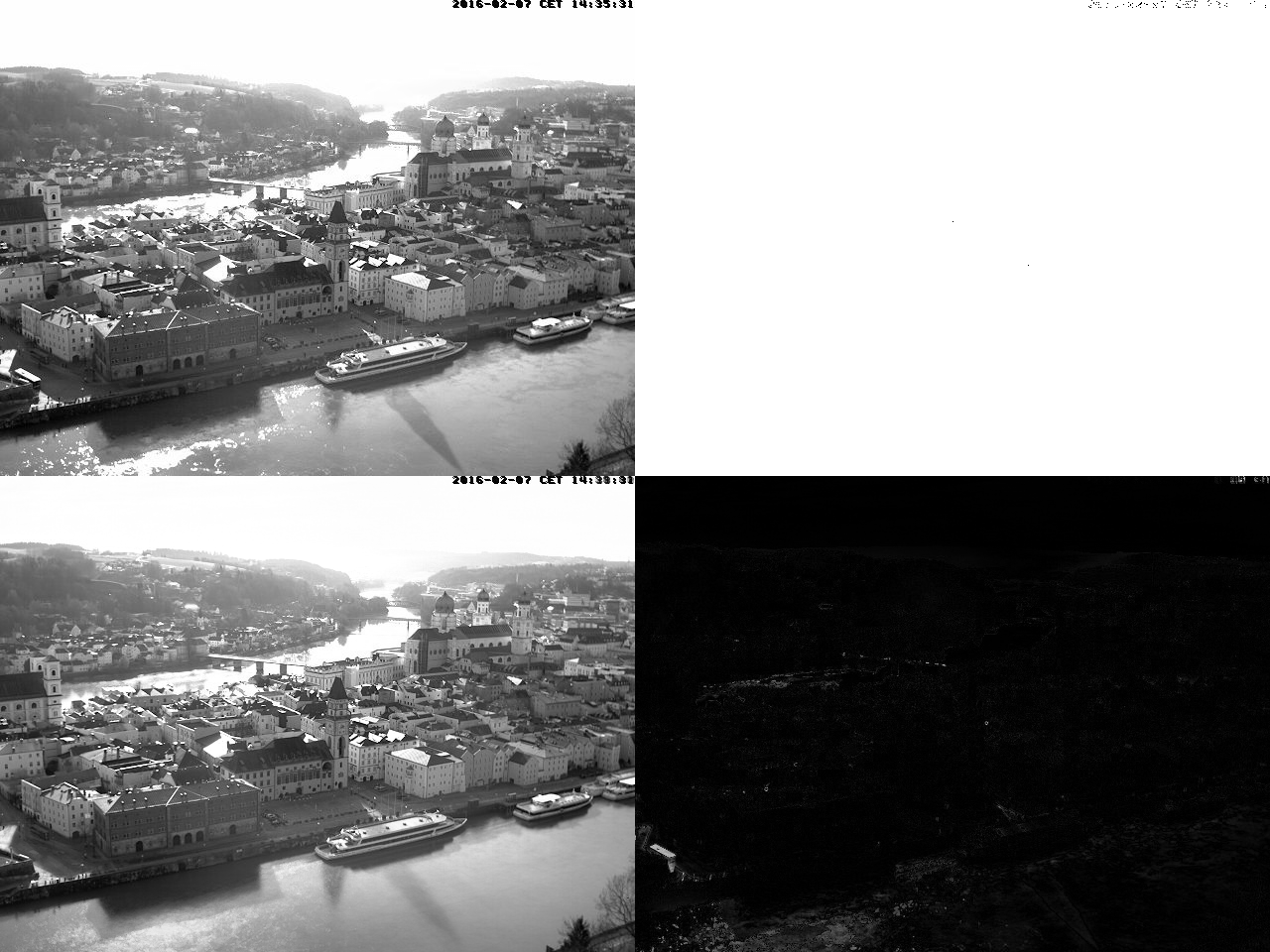}
  \caption{Frame \#507, decomposition: Moving objects (bus) and light
    effects on the water caused by the afternoon sun. Note that the
    shadows are still reproduced in the ``still'' image.} 
  \label{fig:Ripples}
\end{figure}

\subsection{Prony's problem in several variables}
\label{sec:prony}
The main motivation for Algorithm~\ref{A:Augment},
however, was the multivariate version of 
Prony's method \cite{prony95:_essai} which has attracted
some interest recently; besides being
interesting by itself, it is the main mathematical problem behind the
\emph{superresolution} concept from \cite{candes12:_towar}, which in
turn is motivated by studying point spread functions from
microscopy. In a nutshell, \emph{Prony's problem} can be described as
follows: given a function in $s$ variables of the form
\begin{equation}
  \label{eq:PronyProb}
  f (x) = \sum_{\omega \in \Omega} f_\omega \, e^{\omega^T x}, \qquad
  f_\omega \in \CC, \quad \Omega \subset \left( \RR + i ( \RR / 2 \pi
    \ZZ ) \right)^s, \,
  \# \Omega < \infty,  
\end{equation}
recover the unknown \emph{frequencies} $\omega$ from the finite set
$\Omega$ as well as the coefficients $f_\omega$, $\omega \in \Omega$,
from integer samples of $f$, i.e., from $f(A)$, $A \subset
\ZZ^s$. Note that the restriction on the imaginary part of the
frequencies is required to make the solution unique and the problem
well--defined. The main assumption made when solving this problem is
\emph{sparsity}, which means that $\# \Omega$ is small while no other
assumptions on $\Omega$ are necessary, though of course the
conditioning of the problem will depend on the geometry of $\Omega$.

Though determining $\Omega$ is a nonlinear problem, it can be
approached 
by methods from Numerical Linear Algebra. As pointed out in
\cite{Sauer17_Prony,Sauer18:_prony}, the \emph{Hankel matrices}
\begin{equation}
  \label{eq:PronyHankel}
  F_{A,B} := \left[ f( \alpha - \beta) :
    \begin{array}{c}
      \alpha \in A \\ \beta \in B
    \end{array}
  \right], \qquad A,B \subset \NN_0^s,
\end{equation}
provide all information about the ideal
\begin{equation}
  \label{eq:IOmegaDef}
  I_\Omega := \left\{ p \in \CC [x_1,\dots,x_s] : p( e^\omega ) = 0, \,
    \omega \in \Omega \right\}, 
\end{equation}
provided that $A$ and $B$ are sufficiently rich. Once a basis for
$I_\Omega$ is determined, the common zeros and therefore the
frequencies can be determined by methods from Computer Algebra. In
particular, if $A$ is such that the monomials $x^\alpha$, $\alpha \in
A$ admit interpolation at $e^\Omega := \{ e^\omega : \omega \in \Omega
\} \subset \CC^s$, then a polynomial belongs to $I_\Omega$ if and only
if its coefficient vector is in the kernel of $F_{A,B}$. By increasing
$B$ in a proper way, one can so construct Gr\"obner or H--bases for
$I_\Omega$ with which the computation of the frequencies is reduced to
an eigenvalue problem.
The main observation from \cite{Sauer17_Prony,Sauer18:_prony} is now
as follows.

\begin{theorem}\label{T:PronyIdeal}
  If $A \subset \NN_0^s$ is sufficiently rich in the sense that
  $$
  \Pi = I_\Omega + \mbox{\textrm{span}\,} \{ (\cdot)^\alpha : \alpha
  \in A \},
  $$
  then, with $F_{A,B}$ as in (\ref{eq:PronyHankel}),
  \begin{enumerate}
  \item $[ p_\alpha : \alpha \in B ] \in \ker F_{A,B}$ if and only if
    $p(x) = \sum p_\alpha x^\alpha \in I_\Omega$.
  \item if $\rank F_{A,\Gamma_n} = \rank F_{A,\Gamma_{n+1}}$, then
    $\ker F_{A,\Gamma_n} (x)$ is a basis of $I_\Omega$, where
    $\Gamma_n := \{ \alpha \in \NN_0^s : |\alpha| \le n \}$.
  \end{enumerate}
\end{theorem}

\noindent
These two observations suggest the algorithm to solve Prony's problem:
first find a ``good'' set $A$ and then build successively the matrices
$$
F_{A,\Gamma_0}, F_{A,\Gamma_1},\dots,  F_{A,\Gamma_n},
F_{A,\Gamma_{n+1}}, \dots 
$$
until $\rank F_{A,\Gamma_n} = \rank F_{A,\Gamma_{n+1}}$. If
$$
F_{A,\Gamma_k} = U \Sigma V^T,
$$
then the components of $V$ belonging to zero singular values are, by
Theorem~\ref{T:PronyIdeal}, coefficient vectors of polynomials from
the ideal, and once the rank stabilizes, a basis of the ideal has been
found from which the set $\Omega$ can be computed. Hence, in contrast
to the PCA application before, where the singular vectors in $U$ were
of importance, we are now interested in the matrix $V$ and it's
capability to distinguish between the kernel of $F_{A,\Gamma_k}$ and
its orthogonal complement.

There is one major drawback, however: in several variables, the
geometry of $e^\Omega$ becomes increasingly relevant and usually, only
$\# \Omega$ or an upper bound for it are assumed to be known. The
smallest known choice for $A$ that works unconditionally without any
further assumptions on $\Omega$ has cardinality $\# \Omega \, \left(
  \log \# \Omega \right)^{s-1}$ which still grows quite fast for large
$s$. Moreover, it is known to be beneficial to oversample, i.e., to
choose $A$ larger than needed, cf. \cite{Batenkov17P}. Thus,
Algorithm~\ref{A:Augment} addresses the two main issues here: how to
handle large columns in a still efficient way and how to ensure that
the rank is controlled well.

\section*{Acknowledgement}
We want to thank the referee for the very critical but constructive
report that significantly improved the paper and helped us a lot to
clarify the main points of this method. This was exceptionally helpful.

\section*{References}


\end{document}